\documentclass[onecolumn]{IEEEtran}

\usepackage{wrapfig}
                
\usepackage{graphicx}      

\usepackage{amsmath}
\usepackage{amssymb}
\usepackage{hyperref}

 \usepackage{hyperref}
\usepackage{amsmath,amssymb,mathtools}
\usepackage{blkarray, bigstrut}
\usepackage{booktabs}
\usepackage{enumerate}
\usepackage{circuitikz} 

\usetikzlibrary{arrows.meta}

\usepackage{verbatim}

\usepackage{subcaption}
\usepackage{float}
 \usepackage{dsfont}
 
\usepackage{tikz}
\usetikzlibrary{positioning}

 \newcommand{\real}{\operatorname{Re}}
 \newcommand{\imag}{\operatorname{Im}}
 
 \newcommand{\sign}{\operatorname{sgn}}

\newtheorem{theorem}{Theorem}
 
 \newtheorem{assumption}{Assumption}

 \newtheorem{problem}{Problem}
  \newtheorem{proposition}{Proposition}
 \newtheorem{remark} {Remark}

\newcommand{\beq}       {\begin{equation}}
\newcommand{\eeq}       {\end{equation}}
\newcommand{\bery}      {\begin{array}}
\newcommand{\eery}      {\end{array}}
\newcommand{\berys}     {\begin{array*}}
\newcommand{\eerys}     {\end{array*}}
\newcommand{\beqry}     {\begin{eqnarray}}
\newcommand{\eeqry}     {\end{eqnarray}}
\newcommand{\beqrys}    {\begin{eqnarray*}}
\newcommand{\eeqrys}    {\end{eqnarray*}}

\def \R {{\mathbb R}}

 \newcommand{\ii}{\mathrm{i}}

\def\sign{\mathop{\rm sign}\nolimits}

\def\Re{\mathop{\rm Re}\nolimits}

 \def\C{{\mathbb C}}

\def\T{{\mathcal T}}

\newcommand{\be}{\begin{equation}}
\newcommand{\ee}{\end{equation}}

\begin{document}

%
\title{On the Asymptotic
Switching Density  in Time-Optimal Control of Linear Systems}

\author{Omri Dalin, Alexander Ovseevich, and Michael Margaliot%
\thanks{This research is partially supported by  research grant 221/24 from the Israeli Science Foundation (ISF).}
\thanks{The authors are with the School of Electrical and Computer Engineering, Tel Aviv University, Israel (e-mail: michaelm@tauex.tau.ac.il).}}

\maketitle

\begin{abstract}
We study the time-optimal control of a controllable linear system on a   time horizon 
$[0,T]$, focusing on the asymptotic switching density        for large~$T$.  When the system matrix has only real eigenvalues, it is well-known  that the number of switches is upper 
bounded uniformly in 
$T$; when it has complex eigenvalues, no such uniform bound exists, and the switching count instead typically grows with 
$T$.  We characterize this growth for a system matrix with an arbitrary spectrum, allowing simultaneously for real eigenvalues, complex eigenvalues, and  a non-trivial Jordan structure.   If the dominant mode is complex, the number of switches grows at least linearly in 
$T$, with an explicit lower bound  expressed via the mean motion problem and  the Bohl–Weyl–Wintner formula. We illustrate the theory on a linearized aircraft pitch/altitude model, showing close agreement between the predicted asymptotic switching rate and numerically computed time-optimal controls.
\end{abstract}



\section{Introduction}
Analyzing the 
switching function in an 
optimal control problem  is important  for both  theoretical
and computational reasons. If the optimal control is bang-bang and
the number of switches is upper bounded by a   constant~$N$ for any time horizon~$[0,T]$ then the infinite-dimensional problem of finding the optimal control reduces to a finite-dimensional parameter optimization over   up to~$N-1$ 
discrete switching times. More generally, a 
known bound on the number of switching points is   useful  
in  numerical schemes for solving optimal control problems
such as direct shooting or pseudospectral methods (see e.g.~\cite{grognard2003computation} and the references therein). 

Second, bounding the number of switches is essential for establishing ``nice'' reachability results, that is, 
identifying a 
subset of controls~$\mathcal{W}$ with ``favorable'' properties such that steering the system between two states can always be achieved using a control from~$\mathcal{W}$ (see, e.g., \cite{suss_nice_reach,levinson_select,SHARON_3_nilp}). Such reachability results play an important role in proving stability under arbitrary switching 
in switched systems  \cite{MARGALIOT_variational,margaliot_liberzon_lie_algebraic}.

 Third, in implementing optimal controllers, the switching frequency determines how rapidly the actuator must toggle between its saturation limits, so an upper bound on the number of switches provides an upper bound
 on the actuator's switching rate.

Consider the single-input linear control system
\begin{align}\label{eq:lin_cont_sys}
\dot x&=Ax + b  u,
\end{align}
with $x:[0,\infty)\to \R^n$, $A\in\R^{n\times n}$, $b \in\R^{n}$, and measurable controls~$u:[0,\infty)\to[-1,1]$. We assume throughout that~$(A,b)$ is controllable. Fix an 
initial state $x^1\in\R^n$ and a target state $x^2\in\R^n$ such that~$x^2$ is reachable from~$x^1$ using an admissible control, and consider the problem of finding an admissible  control $u$ 
that steers the system from $x(0)=x^1$ to~$x(T)=x^2$ in \emph{minimal time}~$T$.

It is well-known (see e.g. \cite[Ch. 1]{bonnard_chyba}) that an optimal control exists and satisfies $u^*(t)=\sign (m(t))$,
where the switching function $m:[0,T]\to\R$ is defined by 
 $
 m(t):=p^\top (t) b . 
$ 
Here $p:[0,T]\to\R^n\setminus\{0\}$  is the adjoint state vector
of the Pontryagin maximum principle, and it  satisfies
\begin{equation}\label{adjoint}\dot p(t)=- A^\top p(t), 
\end{equation}
so
\be\label{eq:thimi}
m(t)=p^\top(0) e^{-At} b. 
\ee

Under the controllability assumption, $m$ has a finite number of zeros on any open time interval, ensuring any time-optimal control is    bang-bang. Since $m$ is continuous, the optimal control $u^*$ can switch from $-1$ to $1$ (or vice versa) only at points that are zeros of $m$. 

Thus, an important problem   is determining the number of zeros of $m$ in an
interval $[0,T]$, denoted $N(T,0)$. 
A well-known result \cite[Thm.~6-8]{athans_falb} establishes that if all the eigenvalues of~$A$ are real, then~$N(T,0) \leq n-1$ for any~$T>0$. Thus the control sequence is   parametrized by at most $n-1$ switching times $t_1<\dots<t_{n-1}$.

When complex modes dominate the switching function, $N(T,0)$ need not remain uniformly bounded in \(T\). To demonstrate this, consider a simple two-dimensional damped oscillator:
\begin{align*}
    \dot x =\begin{bmatrix}
    -\sigma&\omega\\-\omega & -\sigma 
    \end{bmatrix}x +\begin{bmatrix}
        0\\1
    \end{bmatrix}u,
\end{align*}
with $\sigma \ge 0$, $\omega > 0$. The eigenvalues of the system matrix are complex: $\lambda_1=
-\sigma - i\omega$ and~$\lambda_2= -\sigma+i \omega $. Solving the adjoint equation \eqref{adjoint} yields the switching function:
\begin{align*}
    m(t)  = p_2(t) 
          = a e^{\sigma t} \sin(\omega t + b)
\end{align*}
for   $a, b \in \R$ with $a \neq 0$. The zero-crossings  depend on the periodic term~$\sin(\omega t + b)$, so
the number of switches on~$[0,T]$ grows asymptotically with~$T$ as~$ \omega T/\pi $.

Typically the value of the adjoint state vector  is   not known explicitly, so it is natural to      replace~\eqref{eq:thimi} by  a  ``generic'' 
switching
function
\be\label{eq:abst_m}
m(t;p,b, A):=p^\top e^{-At}b
\ee
with~$p,b\in\R^n\setminus\{0\}$,~$A\in\R^{n\times n}$ and the pair~$(A,b)$ controllable. 
 The zeros of the switching function~$m$  are called the switching points. Since scaling~$m$ by a non-zero constant does not change its zeros, we always assume that~$\|p\|=1$.

 Ref. \cite{dalin2026application}  characterized the generic 
  switching behavior in the
  case where all eigenvalues of the system matrix~$A$ are purely imaginary and simple, and showed
  that the number of switching points is bounded from below by a linear   function of the time horizon $f(T)= c T$, and that $c$ can be determined 
  using the classical mean motion problem~\cite{Favorov2008} 
  (see also~\cite{dalin2026number} that established a similar linear lower bound in this case via an alternative approach based on estimating the $L_1 $ norm of the switching function).
  
 Here, we generalize these ideas to the case where~$A$ has arbitrary spectra.
 In particular, $A$ is not necessarily diagonalizable
 and can have both real eigenvalues and complex eigenvalues, and not only purely imaginary eigenvalues.  The main contributions of this note  include the following. First, we establish a general asymptotic framework for the switching behavior of linear time-optimal control systems, allowing the system matrix to have real and complex eigenvalues and nontrivial Jordan structure. Second, we show that the asymptotic switching behavior is determined by the dominant spectral modes of the switching function, selected successively by their real parts and Jordan degrees, and, in the presence of oscillatory modes, by the mean motion of the corresponding dominant trigonometric polynomial. This yields an explicit asymptotic lower bound on the switching density and, under additional conditions, an exact asymptotic switching rate. The mean motion rate can be  expressed 
 using  the Bohl-Weyl-Wintner formula, providing a one-dimensional integral representation in terms of Bessel functions that facilitates numerical evaluation. The theoretical results are illustrated on a linearized aircraft model.

The remainder of this note  is organized as follows. The next section briefly
reviews the mean motion problem  including the  Bohl–Weyl–Wintner~(BWW) formula, as this plays an important role in the analysis.
For the sake of completeness,   we also  include more details 
on the proof  of  this 
formula   in the Appendix.
Section ~\ref{sec:main} presents the main results.  
 The final section concludes, and describes some directions for further research.

 We use standard notation.
Vectors [matrices] are denoted by small [capital] letters.
The transpose of a matrix~$A$ is~$A^\top$.
 The set of integers is denoted by~$\mathbb{Z}$, and~$\mathbb N_0=\{0,1,2,\dots\}$. 
$\R^n$ [$\C^n$]
is the~$n$-dimensional vector space over the field of real [complex] scalars, and we abbreviate~$\R^1$ to~$\R$ [$\C^1$ to~$\C$]. 
With a slight abuse of notation, we call a number~$z=\alpha +\ii  \beta$, with~$\alpha,\beta\in\R$,
a complex number if~$\beta \not =0$.
We use  $|z|$ [$ \arg (z)$]
to denote 
the absolute value [argument]  of~$z$,
 so the polar representation of~$z$ is
$z=|z| e^{\ii \arg (z)}$. Also,
  $\real(z)$  [$\imag(z)$] denotes the  real [imaginary] part of~$z$,  and~$\bar z=|z| e^{-\ii \arg (z)}$ is the complex conjugate   of~$z$. 

\section{Preliminaries}\label{sec:pri}
We first review the mean motion problem~\cite{Favorov2008}, as this plays a  crucial  role
in our analysis.  

\subsection{Mean Motion Problem}
 Consider the  complex function $z:[0,\infty)\rightarrow\mathbb{C}$ given 
 by:
\begin{equation}\label{eq:rotating}
z(t) = \sum_{k=1}^n a_k e^{\ii (\lambda_k t+\mu_k)},
\end{equation}
where $a_k, \lambda_k, \mu_k \in \R$.
Thus, $z$ is the weighted
sum of the states of~$n$ linear oscillators. 
 
\begin{assumption}\label{assu:znonzero}
    Assume that~$z(t)\not =0$ for all~$t\geq 0$.
\end{assumption}
This assumption guarantees  that
$
\Phi(t):=\arg( z(t))
$
is a continuous function.
Using codimension arguments shows that for given~$\lambda_1,\dots,\lambda_n$, this assumption  holds for a generic set of parameters~$\{a_k,\mu_k\}_{k=1}^n$.\footnote{In this paper, generic  refers to a property that holds for  a dense open subset of parameters.}

\begin{problem}[mean motion]\label{prob:mean_motion}
Determine whether the asymptotic angular velocity of~$z(t)$, that is, the   limit  
 \be\label{eq:def_omeha}
 \Omega:=\lim_{t\to \infty}\frac{\Phi(t)}{t},
 \ee
   exists, and if so,   find  its value.
\end{problem}

For example, if~$\lambda_1=\dots=\lambda_n=\lambda$ then $ z(t) =e^{i \lambda  t}  \sum_{k=1}^n a_k e^{i\mu_k }$, so 
\begin{align*}
   \Phi(t) & =   \arg ( e^{i \lambda  t}  \sum_{k=1}^n a_k e^{i\mu_k }  )  
\\    
    &=    \lambda t+ \arg(\sum_{k=1}^n a_k e^{i\mu_k })  ,
\end{align*}
and thus the mean motion is~$\Omega=\lambda$.

Ref.~\cite{dalin2026application}  showed that Problem~\ref{prob:mean_motion} plays 
an important role 
 in the case where~$A$ has only simple and 
purely imaginary eigenvalues, as then the switching function  
has  the form~$m(t)=\sum_{k=1}^n
    a_k\cos(\lambda_k t+\mu_k)$, that is, 
   $ 
    m(t)=  \real(z(t)) , 
    $
   with $z $ defined in~\eqref{eq:rotating}. 
    If~$z$ has mean motion~$\Omega$ then 
   asymptotically the real part of~$z$ will have at least~$2 \frac{|\Omega | }{2\pi}  T$ zeros on~$[0,T]$, so~$m$ will have at least $  \frac{|\Omega|   }{ \pi} T $ zeros on~$[0,T]$.

\begin{assumption}\label{assumption:almost_non_resonance}
   The {\em almost non-resonance} condition for the function~$z(t)$ in~\eqref{eq:rotating}   is that 
   there  are no non-trivial relations of the form 
  \be
  \sum_{k=2}^n ( \lambda_k-\lambda_1) \ell_k=0, \text{ where every }
\ell_k \text{ is an  integer}.
\ee
\end{assumption}

\subsection{Bessel functions and the Bohl-Weyl-Wintner formula}
The solution of the mean motion problem
includes two  Bessel functions: 
\[
J_0(x) =  \frac1{\pi}  \int_0^\pi  e^{-\ii  x\cos(\phi)} d\phi,
\]
and
\[
J_1(x) = \frac { x   }{   \pi} \int_0^\pi \sin^{2 }(\phi)e^{- \ii x\cos(\phi)} d\phi.
\]
Given angles $\phi_1,\dots,\phi_n\in[0,2\pi)$, let  
\begin{align*}
    z(\phi_1,\hdots,\phi_n):=\sum_{k=1}^na_k e^{\ii \phi_k},
\end{align*}
This can be interpreted as the position of the endpoint of a multi-link robotic arm, where the $k$th link has length $|a_k|$, and the angle between links \(k\) and \(k+1\) is \(\phi_k\).

The canonical (Haar)  probability measure on the $n$-torus is
$
d\mu=\frac1{(2\pi)^n} d\phi_1\dots d\phi_n.
$
Let~$W_n(r)=W_n(r;a_1,\dots,a_n)$ denote the probability 
that~$|z|\leq r$, that is,
\begin{align}\label{eq:W_n(r)}  
W_n(r) =
\frac{1}{(2\pi)^n}\int_{ |z|\leq r }d\phi_1\cdots d\phi_n.
\end{align}
This is the ``volume'' of angles $\phi_1,\dots,\phi_n$ such that the position of the multi-link robotic arm is in the ball of radius~$r$ centered at the origin.

The BWW formula \cite{Wintner-1933} asserts  that \begin{equation}\label{BWW}
W_n(r;a_1,\dots,a_n)=r\int_0^\infty J_1(r\rho)\prod_{k=1}^n J_0(|a_k|\rho)d\rho.
\end{equation}
Note that this
reduces the computation of the $n$-dimensional integral for~$W_n$ in~\eqref{eq:W_n(r)}
to the one-dimensional integral in~\eqref{BWW}. 
\subsection{Solution of the mean motion problem}
  Weyl \cite{Weyl_meanmotion} proved that under Assumption~\ref{assu:znonzero} the mean motion~$\Omega$ in~\eqref{eq:def_omeha}   exists. If  Assumption~\ref{assumption:almost_non_resonance} also holds, then 
\begin{equation}\label{mean}
    \Omega=\sum_{k=1}^n \lambda_k V_k ,
\end{equation}
where
$
V_k : =W_{n-1}(a_k;a_1,\dots,a_{k-1},a_{k+1},\dots,a_n).
$
Furthermore,
the  $V_k$s  are non-negative, and $\sum_{k=1}^n V_k=1$ (see~\cite{Weyl_meanmotion}), so~\eqref{mean} is a convex combination  of the~$\lambda_k$s.
The BWW formula then gives a  closed-form integral
expression for the mean motion, namely, 
\begin{align*}
    \Omega=\sum_{k=1}^n\lambda_k a_k\int_0^\infty J_1(a_k\rho)\prod_{\ell\neq k}^{n}J_0\left(|a_\ell|\rho\right)d\rho.
\end{align*}

\section{Main Results}\label{sec:main}
We begin with an auxiliary result that will be used later on to analyze the switching function where~$A$ has both real and complex eigenvalues. 

\subsection{Asymptotic switching bound for polynomially weighted oscillations}

\begin{proposition} \label{prop:asw}
Let
$
m:[0,\infty)\to\mathbb R
$
be given by
\begin{equation}
    \label{eq:generalm}
m(t)
=
e^{\alpha t}
\left[
\sum_{j=1}^{q}\sum_{\ell=0}^{d_j}
H_{j,\ell}t^\ell
\cos(\Omega_jt+\psi_{j,\ell})
+
R(t)
\right],
\end{equation}
with
$
\alpha\in\R$, $
\Omega_j, H_{j,\ell} \geq 0$, $
d_j\in\mathbb N_0$,
and~$R(t) $ is a $C^1$ function. 
Let
$$
d_*:=\max\bigl\{\ell:H_{j,\ell}\neq0
\text{ for some }j\bigr\},
$$
and define the \emph{dominant oscillatory component} of~$m$ by 
$$
Q(t):=
\sum_{j=1}^{q}H_j
\cos(\Omega_jt+\psi_j),
\qquad
H_j:=H_{j,d_*}. 
$$
Assume that
\begin{enumerate}[(a)]
    \item the complex function
$$
z(t):=
\sum_{j=1}^{q}H_j e^{i(\Omega_jt+\psi_j)}
 ,
$$
is nonzero for all~$t\geq 0$, and   
 its  mean motion  
$$
\Omega
:=
\lim_{t\to\infty}
\frac{\arg z(t)}{t}.
$$
exists;
\item there exists \(\gamma>0\) such that
\begin{equation}
    \label{eq:Q+Q'}
|Q(t)|+|\dot Q(t)|\geq\gamma,
\text{ for all } t\geq 0; 
\end{equation}
\item all the terms in~$m$ not contained in \(Q\) are asymptotically
smaller than \(t^{d_*}\),   that is, the function 
$$
P(t):=
t^{-d_*}
\sum_{j=1}^{q}\sum_{\ell=0}^{d_j-1}
H_{j,\ell}t^\ell
\cos(\Omega_jt+\psi_{j,\ell})
+
t^{-d_*}R(t),
$$
satisfies 
\begin{equation}
P(t)\longrightarrow0 
\text{ and } 
\dot P(t)\longrightarrow0,
\text{ as } t\to\infty.
\label{A2}
\end{equation}
\end{enumerate}
Then there exists \(t_*\geq0\) such that, for all sufficiently large~\(T>t_*\), the number \(N(T,t_*)\) of zeros of \(m\) on
\([t_*,T]\) satisfies
\begin{equation}
N(T,t_*)
\geq
\frac{|\Omega|}{\pi}(T-t_*)
+
o(T-t_*).
\label{A3}
\end{equation}
In particular,
\begin{equation}
\liminf_{T\to\infty}
\frac{N(T,t_*)}{T}
\geq
\frac{|\Omega|}{\pi}.
\label{A4}
\end{equation}
\end{proposition}

Before proving this result, we make  several comments.
First, the structure in~\eqref{eq:generalm} 
covers in particular  the 
  case 
of 
a dominant Jordan block of size \(k+1\) in the system matrix~$A$, as this produces terms of the form
$
  e^{\alpha t} t^\ell  \cos(\omega t+\xi)$,
  $ \ell= 0,1,\dots,k$, 
in the switching function. 

Second, note that if
  the mean motion of~$z$  is expressed through the BWW
weights as
$$
\Omega=\sum_{j=1}^{q}\Omega_jV_j,
\qquad
V_j\geq0,\qquad
\sum_{j=1}^{q}V_j=1,
$$
then \eqref{A4}
immediately yields the lower bound 
$$
\liminf_{T\to\infty}
\frac{N(T,t_*)}{T}
\geq
\frac{1}{\pi}
\left|
\sum_{j=1}^{q}\Omega_jV_j
\right|.
$$

Third, consider the case~$q=1$ and~$d_1=0$, that is,
$Q(t)=H_1\cos(\Omega_1 t +\psi _1)$. Then the
complex function is~$z(t)=e^{  \ii (\Omega_1 t+\psi_1) } $, 
so the mean motion is just~$\Omega=\Omega_1$. Note that 
\begin{align*}
    |Q(t)|+|\dot Q(t)|  & = H_1 |\cos(\Omega_1t+ \psi_1)|+ H_1 | \sin(\Omega_1t+ \psi_1)|, 
\end{align*}
so
$
  (  |Q(t)|+|\dot Q(t)| )^2 >   H_1^2
$
implying that~\eqref{eq:Q+Q'} indeed holds for $\gamma=H_1$. 
More generally, 
the uniform transversality condition~\eqref{eq:Q+Q'}
is sufficient to guarantee that the zeros of \(Q\) are uniformly simple. For periodic \(Q\), this condition follows whenever \(Q\) and \(\dot Q\) do not vanish simultaneously. For 
quasiperiodic~\(Q\), uniform transversality is a stronger requirement and is therefore imposed explicitly.

\begin{IEEEproof}[Proof of Prop.~\ref{prop:asw}]
Since
$
m(t)
=
e^{\alpha t}t^{d_*}
(Q(t) +P(t) )   ,
$
and
$
e^{\alpha t}t^{d_*}>0
 \text{ for }t>0,
$
the zeros of \(m\) on \((0,\infty)\) are precisely the zeros of
$
\widetilde m(t):=Q(t)+P(t).
$
 Fix $\delta \in (0, \gamma/2)
$.
By \eqref{A2}, 
there exists \(t_1  \geq0\) such that
\begin{equation}
|P(t)|<\delta 
\text{ and } 
|\dot P(t)|<\delta,
\text{ for all } t\geq t_1.
\label{A6}
\end{equation}
Define the sets
$$
\mathcal T_{\rm far}
:=
\{t\geq t_1:|Q(t)|\geq\delta\},
$$
and
$$
\mathcal T_{\rm near}
:=
\{t\geq t_1:|Q(t)|<\delta\}.
$$
For \(t\in\mathcal T_{\rm far}\), \eqref{A6} gives
$
|\widetilde m(t)|
\geq
|Q(t)|-|P(t)|
>
0.
$
Hence \(\widetilde m\), and therefore \(m\), has no zero in
\(\mathcal T_{\rm far}\).

Now let \(I=(a,b)\) be a connected component of
\(\mathcal T_{\rm near}\). For every \(t\in I\), assumption \eqref{eq:Q+Q'}  
gives
$$
|\dot Q(t)|
\geq
\gamma-|Q(t)|
>
\gamma-\delta
>
 {\gamma}/{2}.
$$
Since \(|\dot P(t)|<\delta<\gamma/2\), it follows that
$$
|\dot{\widetilde m}(t)|
\geq
|\dot Q(t)|-|\dot P(t)|
>0  , 
$$
so
 \(\widetilde m\) is strictly monotone on \(I\), and consequently
it has at most one zero in \(I\). The same argument shows that \(Q\) also has
at most one zero in \(I\).

Divide the time interval  $[t_1,\infty)$
as the union of intervals from~$ \mathcal{T}_{far} $, 
and maximal intervals~$I_k:=(a_k,b_k) \subseteq  \mathcal{T}_{near} $. Maximal here means that each~$I_k$ is not strictly contained in a larger interval in $ \mathcal{T}_{near} $. We already saw that on intervals  in~$ \mathcal{T}_{far} $ both~$Q$ and~$\tilde m$ have no switching points, so we can ignore them. Consider a maximal interval~$I_k = (a_k,b_k) \subset  \mathcal{T}_{near} $. 
We analyze
the number of switching points in the following possible cases:

 \noindent{Case 1.} 
If~$b_k-a_k=\infty$ then 
$Q$ has up to a single switching on this infinite  interval,
and so does~$\tilde m$.

\noindent{Case 2.} 
Suppose that~$Q(a_k^-)=\delta$ and~$Q(b_k^+)=-\delta$ (or vice versa). Then $Q$ changes sign on~$I_k$ and monotonicity implies that it  has a single switching in~$I_k$. Also, 
\[
\tilde m(a_k) = Q(a_k)+P(a_k)=\delta +P(a_k)>0,
\]
and similarly~$\tilde m(b_k)<0$, so $\tilde m$  also
has a single switching in~$I_k$.

\noindent{Case 3.} 
 Suppose that~$Q(a_k^-)= Q(b_k^+)= \delta$ (or both are equal to~$-\delta$). Then monotonicity implies that $Q(t)=\delta$ for all~$ t\in I_k$, so in particular  it has zero switching points  
 in the interval. Arguing as in Case~2 shows that $\tilde m$ also has zero switching points  
 in the interval.

Thus, after some finite time \(t_*\geq t_1\), there is a one-to-one
correspondence between the zeros of \(m\) and the zeros of \(Q\). Hence
\begin{equation}
N(T,t_*)
=
N_Q(T,t_*)+O(1),
\label{A7}
\end{equation}
where \(N_Q(T,t_*)\) denotes the number of zeros of \(Q\) on
\([t_*,T]\).
Since
$ 
Q(t)=\Re( z(t)),
 $
the mean-motion zero-count result for the real part of \(z\)  
gives

\begin{equation} 
N_Q(T,t_*)
\geq
\frac{|\Omega|}{\pi}(T-t_*)
+
o(T-t_*).
\label{A8}
\end{equation}
Combining \eqref{A7}
and~\eqref{A8}
yields~\eqref{A3}, and
division by~\(T\) gives~\eqref{A4}.
This completes the proof of Prop.~\ref{prop:asw}.
\end{IEEEproof}


\subsection{Asymptotic switching density  for a general linear system}

Consider the  linear system~\eqref{eq:lin_cont_sys}
with~\(A\in\mathbb R^{n\times n}\)  and~\((A,b)\)  controllable.
Since the  bound
in the case where $A$ has only real eigenvalues  is well-known, we assume that $A$ has at least one complex eigenvalue. 
Let~\eqref{eq:abst_m} 
be a switching function associated with a Pontryagin extremal.
A
    Jordan decomposition of~\(-A\) 
implies that 
\begin{equation}
    \label{eq:mstru}
m(t)
=
\sum_{\nu}
e^{\alpha_\nu t}t^{d_\nu}
H_\nu\cos(\omega_\nu t+\phi_\nu),
\end{equation}
with \(H_\nu > 0\),~$\omega_\nu\geq 0$ and at least one $\omega_v  $ is positive. Here,  terms with \(\omega_\nu=0\) represent the contributions of
real eigenvalues.  

Define the \emph{active dominant real part} of~$m$  by 
$$
\alpha_*:=\max_\nu \{\alpha_\nu \} 
$$
and, among the terms with real part \(\alpha_*\), define the \emph{dominant
Jordan degree} by 
$
d_*:=
\max\{d_\nu:\alpha_\nu=\alpha_*\}.
$
Let
$$
\mathcal D:=
\{\nu:\alpha_\nu=\alpha_*,\ d_\nu=d_*\},
$$
and define the \emph{dominant oscillatory component} of~$m$  by 
\begin{equation}
Q(t):=
\sum_{\nu\in\mathcal D}
H_\nu\cos(\omega_\nu t+\phi_\nu).
\label{eq:domQ}
\end{equation}

 We can now give the  main result in this section.  
 \begin{theorem} \label{thm:main} 
The switching  function satisfies 
\begin{equation}
    \label{eq:m_simpi}
m(t)
=
e^{\alpha_*t}t^{d_*}
\bigl(Q(t)+P(t)\bigr),
\end{equation}
where
\begin{equation}
    \label{eq:P_DECAY}
P(t)\longrightarrow0 
\text{ and }
\dot P(t)\longrightarrow0
\text{ as } t\to\infty.
\end{equation}
Suppose that:
\begin{enumerate}[(a)]
    \item 
\(Q\not\equiv0\), and   there exists
\(\gamma>0\) such that
\be \label{eq:qcond}
|Q(t)|+|\dot Q(t)|\geq\gamma,
\text{ for all } t\geq0.
\ee 
\item 
the complex function 
$
z(t):=
\sum_{\nu\in\mathcal D}
H_\nu e^{\ii (\omega_\nu t+\phi_\nu)},
$
satisfies 
$
z(t)\neq0 $
  for all  $t\geq0,
$
and   the mean motion
$
\Omega:=
\lim_{t\to\infty}
\frac{\arg z(t)}{t}
$
exists.
\end{enumerate}
Then there exists \(t_*\geq0\) such that the number~\(N(T,t_*)\)
of switching points of~\(m\) on~\([t_*,T]\) satisfies

$$
N(T,t_*)
=
N_Q(T,t_*)+O(1),
\label{9}
$$
where \(N_Q(T,t_*)\) is the number of zeros of \(Q\) on
\([t_*,T]\). In particular,
$$
N(T,t_*)
\geq
\frac{|\Omega|}{\pi}(T-t_*)
+
o(T-t_*),
\text{ as } T\to\infty,
\label{10}
$$
and hence
$$
\liminf_{T\to\infty}
\frac{N(T,t_*)}{T}
\geq
\frac{|\Omega|}{\pi}.
$$
\end{theorem}

If, in addition, the frequencies in \(z\) satisfy the almost
non-resonance condition required by the mean motion theorem, then
$
\Omega $ 
can be expressed using the BWW formula. 

This result shows in 
 particular that  the asymptotic switching behavior of~$m$  is determined first
by the dominant real part~\(\alpha_*\), then by the largest Jordan
degree~\(d_*\), and finally by the mean motion of the corresponding
 component~\(Q\).

Note also  that if~$\omega_\nu=0$ for all~$\nu\in\mathcal D$  (i.e. no complex eigenvalue is dominant) 
then $Q(t) $ is a constant function, and if it is not zero then~\eqref{eq:qcond} holds, the function~$z(t)$ is also constant  and thus $\Omega=0$.

\begin{IEEEproof}
 By definition of \(\alpha_*\) and \(d_*\), every term in~\eqref{eq:mstru} belonging
to \(\mathcal D\) has the common factor~$
e^{\alpha_*t}t^{d_*},
$
whereas every remaining term is smaller after division by this factor.
Thus,
$
m(t)
=
e^{\alpha_*t}t^{d_*}
\bigl(Q(t)+P(t)\bigr),
$
with \(Q\) given by~\eqref{eq:domQ}.
To verify~\eqref{eq:m_simpi}, consider a term corresponding to~\(\alpha_\nu=\alpha_*\) and~\(d_\nu<d_*\). After division by~\(e^{\alpha_*t}t^{d_*}\), such a term becomes
$
O(t^{d_\nu-d_*})=O(t^{-1}) , 
$
and
its derivative is also \(O(t^{-1})\). Next, consider a term with~\(\alpha_\nu<\alpha_*\). Writing~$
\delta_\nu:=\alpha_*-\alpha_\nu>0,
$
its normalized contribution is~$
O\left(t^{d_\nu-d_*}e^{-\delta_\nu t}\right),
$
and its derivative satisfies the same type of bound. Thus every
non-dominant term, together with its derivative, converges to zero.
This proves~\eqref{eq:P_DECAY}.
Since
$
e^{\alpha_*t}t^{d_*}>0
$ for all~$t>0$,
the zeros of~\(m\) coincide with the zeros of
$ 
\widetilde m(t):=Q(t)+P(t) ,
$ 
and using Prop.~\ref{prop:asw}
completes the proof. 
\end{IEEEproof}

\begin{remark}
    Consider the 
  case of a unique dominant Jordan mode. Then 
$
m(t)=\exp(\alpha t)
\left ( At^k\cos(\Omega t+\psi) +P(t) \right ) , 
$
with~$
P(t)=o(t^k)$. 
Note that  
\[
m(t)=\exp(\alpha t)t^k
\left (A\cos(\Omega t+\psi) +t^{-k}P(t) \right ), 
\]
so the zeros  asymptotically satisfy
\[
t_j = \frac{ \frac{\pi}{2} -\psi+j \pi }{\Omega} +O(1).
\]
Thus, 
in this case we actually get the stronger result: 
$
\lim_{T\to\infty} \frac{N(T,0)}{T} =\frac{\Omega}{\pi} . 
$
\end{remark}

\section{An Application}
We demonstrate the theoretical
results using an example from~\cite[Ch. 3]{sontag_book_control_theory}. Consider a linearized model for an airplane 
flying at an altitude~$h(t)$ [meters],
constant speed   $c>0$ [meters/sec],
and  pitch angle~$\phi(t)$ [radians] 
with respect to the horizontal. 
Its flight path forms an angle of~$\alpha(t)$ [radians]
 with the horizontal  i.e., for $\alpha(t)>0$ [$\alpha(t)<0$]
 the plane is gaining altitude[ [descending].
 
The mathematical  model is
\begin{align*}
    \dot \alpha &= a_1(\phi  -\alpha),\\
\ddot \phi &=   -\omega^2 (\phi -\alpha - a_2 u ),\\
\dot h  &= c \alpha ,
\end{align*}
where~$\omega >0$ 
is a constant representing a natural oscillation frequency and 
$a_i>0$. 
The control $u$ is proportional to the
position of the elevators (see~\cite[Chapter 3]{sontag_book_control_theory}).  

Letting
$x_1 = \alpha$,
$x_2 = \phi$, $x_3 =\dot\phi$, and $x_4 = h$ yields~\eqref{eq:lin_cont_sys} with
$A=\begin{bmatrix}
    -a_1 & a_1 &0 &0 \\
    0& 0& 1& 0\\
    \omega^2 & -\omega^2 & 0 & 0  \\ 
    c&0 &0 &0 
\end{bmatrix},
$ and $b=\begin{bmatrix}
 0\\0\\ a_2 \omega^2 \\0     
\end{bmatrix}$.
Here
$
    \det ( \begin{bmatrix}
        b&Ab&A^2 b&A^3 b
    \end{bmatrix})=-a_1^2 a_2^4
    \omega^8 c,
$
so the system is controllable.

The eigenvalues of~$A$ are $0,0, (- a_1  \pm \sqrt{ a_1^2-4\omega^2 } )/2 $. We assume that~$a_1<2 \omega$, so $A$ has two complex eigenvalues besides the repeated zero eigenvalue. 
Eq.~\eqref{eq:abst_m} gives
$
m(t) =c_0+c_1 t+  \exp(a_1 t/2) \cos(\Omega t+\phi),
$
with $\Omega =\sqrt{\omega^2-\frac{a_1^2}{4}}$.
Note that
$
m(t)=  \exp(a_1 t/2)  ( Q(t) + P(t)   ), 
$
with~$P(t)=(c_0+c_1 t)\exp(-a_1t/2)$ and~$Q(t) = \cos(\Omega t+\phi)$.
Thus, $m(t)$ has the same  switching points as~$\tilde m(t):=Q(t)+P(t)$. 
By Thm.~\ref{thm:main},
there exists   $t^*>0$ such that for all $T\gg t^*$, we have 
\begin{align*}
        N(T,t^*)=\frac{\Omega}{\pi}(T-t^*)+o(T-t^*) . 
\end{align*}

As a specific numerical
example,  take
$
a_1=1,\; \omega=2,\;  a_2=1/4,\;  c=1,
$ 
initial condition~$x(0)=0$, 
and control~$u(t)\in[-1,1]$. We used MATLAB to compute PMP extremals for several $T$ values in the range~$[1,15]$
   and counted the number of switchings~$N(T,0)$. 
Fig.~\ref{fig:plane}  shows $N(T,0)$  and the
line~$
\frac{\Omega}{\pi}  T   = \frac{\sqrt{15} }{2\pi} T.
$
The theory predicts that for large~$T$, the number of switches follows the 
line with slope~$\frac{\Omega}{\pi}   $.
It may be seen that this indeed holds. 

\begin{figure}[t]
\centering
  \includegraphics[scale=0.45]{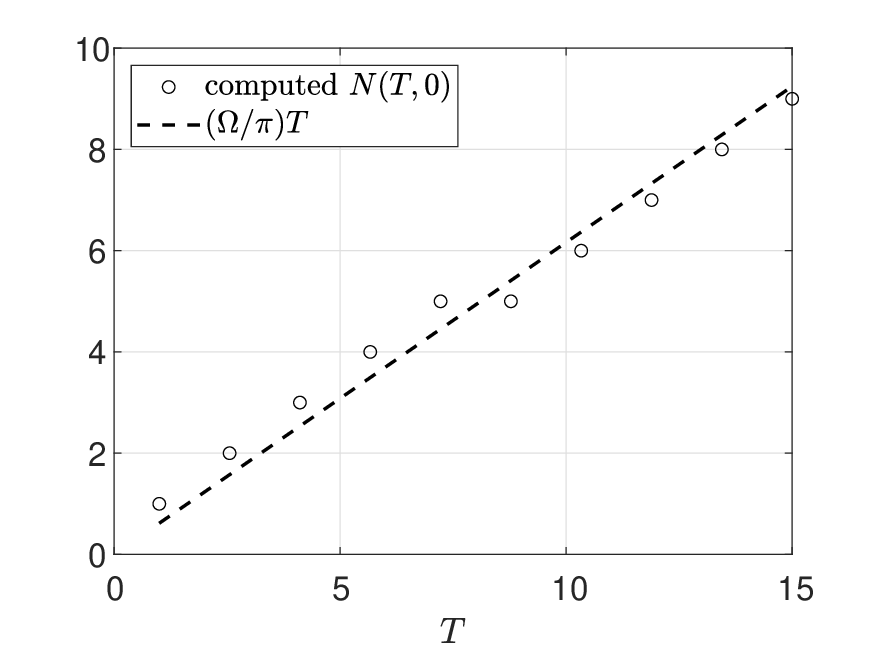}
  \vspace{-2mm} 
  \caption{Number of switchings $N(T,0) $ 
  in time-optimal controls~circles) and the asymptotic linear bound   predicted by the theory~(dashed line), as a function of the final time~$T$.} 
  \label{fig:plane}
\end{figure}

\section{Discussion}
 We characterized the asymptotic switching behavior of time-optimal controls for linear systems with an arbitrary spectrum, extending prior results that were restricted to simple and purely imaginary eigenvalues. The  asymptotics are governed lexicographically: first by the dominant real part $\alpha^*$ of the eigenvalues contributing to the switching function, then by the size~$d^*$ of the corresponding dominant Jordan block, and finally, when~$\alpha^*$ corresponds to a complex mode, by the mean motion  of the resulting dominant oscillatory component.
  
When the dominant complex mode is a single, non-repeated conjugate pair, the asymptotic count becomes   exact. 
Determining other  conditions under which this lower bound is in fact tight is a natural research  question.

\subsubsection*{Acknowledgments} We used ChatGPT to assist with proofreading and checking the manuscript.
 \subsection*{Appendix: proof of the BWW formula }
The~$p$-order Bessel function~$J_p:\mathbb{C}\to \mathbb{C}$  is defined by 
\be\label{eq:p-bessel}
J_p(x) := \frac {(x/2)^p}{  \Gamma(p+\frac1{2}) \sqrt{\pi}}\int_0^\pi \sin^{2p}(\phi)e^{-\ii   x\cos(\phi)} d\phi,
\ee
where~$\Gamma$ is the Gamma function. 
 We first relate Bessel functions to spheres. Let~$1_{B^n}:\R^n\to\{0,1\}$ denote the  indicator function of the $n$-dimensional ball, that is,~$1_{B^n}(x)=1$ if~$|x|\leq 1$ and~$1_{B^n}(x)=0$, otherwise.  The Fourier transform of~$1_{B^n}$  is
\begin{align*} 
\hat 1_{B^n}(\xi)&= \frac1{(2\pi)^{n/2}}
\int_{|x|\leq 1} e^{\ii x^\top  \xi } dx_1\dots dx_n\nonumber\\& = \frac1{(2\pi)^{n/2}}
\int_{|x|\leq 1} e^{-\ii x^\top  \xi } dx_1\dots dx_n . 
\end{align*}
Since~$1_{B^n}(x)$ is a radial function (that is, it only depends  on~$|x|$),
$\hat 1_{B^n}(\xi)$ will   depend on~$|\xi|$   \cite[Chapter~IV]{Four_intro}  and it is
  sufficient    to consider the particular
  case where~$\xi=\begin{bmatrix}
0&\dots&0&r
\end{bmatrix}^\top$,
with~$r>0$. Then,
\begin{align*}
  \hat 1_{B^n}(\xi)&
 = \frac1{(2\pi)^{n/2}} \int_{|x|\leq 1}  e^{-\ii  x_n r} dx_1 \dots dx_n\\
  &=  \frac { V_{n-1}}{(2\pi)^{n/2}}
  \int_{-1}^1 (\sqrt{ 1-x_n^2})^{n-1 }  e^{-\ii x_n r} dx_n,
\end{align*}
where~$V_{n-1} $ is the volume of the~$(n-1)$-dimensional unit ball.
Setting~$x_n= \cos (\phi)$ and using the fact  that~$V_d=\frac{\pi^{d/2}} { \Gamma(1+(d/2)) } $ gives
   \begin{align*}
  \hat 1_{B^n}(\xi)
  &=  \beta  \int_{0}^\pi \sin^{n}(\phi)    e^{-\ii  \cos(\phi)  r} d \phi ,
\end{align*}
with $\beta:= ( 2 ^{n/2} \Gamma(1+((n-1)/2))\sqrt{\pi})^{-1} $, 
and using~\eqref{eq:p-bessel} yields
$
 \hat 1_{B^n}(\xi)  = |\xi|^{-n/2}J_{n/2}(|\xi|).
 $
For~$n=2$, we get 
\be\label{eq:ind_func_B2}
 \hat 1_{B^2}(\xi)  = |\xi|^{-1}J_{1}(|\xi|).
 \ee
Let~$\xi:=\sum_{k=1}^n  a_ke^{i\phi_k}\in \mathbb{C}$.
Fix~$r>0$.
By definition,
\be\label{eq:bydef_W}
W_n(r)=\frac1{(2\pi)^n}\int_{\T} 1_{|\xi|\leq r}(\xi)  \prod_{k=1}^n d\phi_k . 
\ee
Writing
$1_{|\xi|\leq r}(\xi)=1_{B^2}(\xi/r),
$ and applying an inverse Fourier transform to~\eqref{eq:ind_func_B2} yields
\begin{align*}
1_{|\xi|\leq r}(\xi)&=1_{B^2}(\xi/r)
=\frac{1}{2\pi}\int_\C |\eta|^{-1} J_1(|\eta|)e^{\ii \Re  (\bar\eta \xi/r )} d\eta .
\end{align*}
Defining the integration variable
  $\theta:=\eta/r$ gives
\[  
1_{|\xi|\leq r}(\xi)= \frac{r}{2\pi}\int_\C |\theta|^{-1} J_1(r|\theta|)e^{\ii \Re  (\bar\theta \xi )} d\theta.
\] 
Substituting  this in~\eqref{eq:bydef_W} yields
\begin{align*}
W_n(r)&=\frac{r}{ 2\pi }\int_{\T}  \int_\C |\theta|^{-1} J_1(r|\theta|)e^{\ii \Re  (\bar\theta \sum_{k=1}^n a_ke^{i\phi_k}  )} d\theta\\&\times \frac1{(2\pi)^n} \prod_{k=1}^n d\phi_k\\
&= \frac{1}{ 2\pi }\frac{1}{(2\pi)^n}\int_{\T}  \int_\C |\theta|^{-1} J_1(r|\theta|)
\\&
\times
\left (  \prod_{k=1}^n e^{\ii \Re  (\bar\theta  a_k e^{i\phi_k}  )}d\phi_k\right)  d\theta  .
\end{align*}
Write~$ \theta$ in the polar representation~$   \theta=\rho e^{\ii \alpha }$, where $\rho:=|\theta|$. Then
\begin{align*}
 \frac1{2\pi}  \int_0^{2\pi } e^{ \ii \Re  (\bar\theta  a_k e^{i\phi_k}  )}d\phi_k &=\frac1{2\pi}
  \int_0^{2\pi }  e^{ \ii  \Re  ( \rho a_k   e^{\ii (\phi_k-\alpha)}  )}d\phi_k\\
  &=\frac1{2\pi}  \int_0^{2\pi }  e^{\ii  \rho \Re  (  a_k   e^{i \beta_k  }  )}d\beta_k\\
  &= \frac1{2\pi}  \int_0^{2\pi }  e^{\ii  \rho   |a_k|  \cos(\beta_k)  }d\beta_k\\
  &= J_0(\rho |a_k|).
\end{align*}
Thus,
\begin{align*}
W_n(r)
&= \frac{1}{ 2\pi }   \int_\C |\theta|^{-1} J_1(r|\theta|)\left (  \prod_{k=1}^n J_0(|\theta| |a_k|)   \right)  d\theta  .
\end{align*}
 The integrand
$
 |\theta|^{-1} J_1(r|\theta|) \prod_{k=1}^n J_0(|\theta| |a_k|)
 $
 is a radial function, and using known results on the integration of radial functions (see, e.g.,~\cite[Chapter~6]{real_analysis_stromberg})
gives
\begin{align*}
W_n(r)
&=  r \int_0^\infty  J_1(r \rho ) \left (  \prod_{k=1}^n  J_0(\rho |a_k|)    \right)  d\rho,
\end{align*}
and this completes the proof of the BWW formula.

\bibliographystyle{IEEEtran}
 \bibliography{omri}

@string{CDC26="Proc.\ 26th IEEE Conf. on Decision and Control"}

@string{SCL="Systems Control Lett."}

@string{JDE="J. Diff. Eqns."}

@string{IEEE="Proc. IEEE"}

@INPROCEEDINGS{suss_nice_reach,
  author={Sussmann, H. J.},
  booktitle=CDC26, 
  title={Reachability by means of nice controls}, 
  year={1987},
  volume={26},
  pages={1368-1373},
}

@string{j="{\bf Journal version available}"}

@article{Weyl_meanmotion,
author={Hermann Weyl},
title={Mean Motion},
journal={Amer. J Math.},
volume={60},
pages = {889-896},
year = {1938},
}

@book{athans_falb,
 title="Optimal Control:
An Introduction to the Theory and Its Applications",
author="Michael Athans and Peter L. Falb",
publisher="McGraw-Hill Book Company",
year="1966",
}

@book{Four_intro,
    title = "Introduction to Fourier Analysis on Euclidean Spaces",
    author = "E. M. Stein and G. Weiss",
    publisher = "Princeton University Press", 
    address="Princeton, NJ",
    year = "1990"
}

@article{dalin2026number,
  title={On the Number of Switching Points in Time-Optimal Controls of a Linear System},
  author={Dalin, Omri and Ovseevich, Alexander},
  journal={Proceedings of the Steklov Institute of Mathematics},
  volume={332},
  pages={75--80},
  year={2026},
  publisher={Pleiades Publishing}
}

@article{dalin2026application,
    author = "Omri Dalin and Alexander Ovseevich  and Michael Margaliot",
    title = "An Application of the Mean Motion Problem
to Time-Optimal Control",
    journal = "IEEE Control  Systems Letters", volume={10}, pages={307-312}, 
    year = "2026",
}

@inproceedings{grognard2003computation,
  title={Computation of time-optimal switchings for linear systems with complex poles},
  author={Grognard, Fr{\'e}d{\'e}ric and Sepulchre, Rodolphe},
  booktitle={2003 European Control Conference (ECC)},
  pages={2190-2195},
  year={2003},
  organization={IEEE}
}

@article{MARGALIOT_variational,
title = {Stability analysis of switched systems using variational principles: An introduction},
journal = {Automatica},
volume = {42},
number = {12},
pages = {2059-2077},
year = {2006},
author = {Michael Margaliot}
}

@article{Favorov2008,
    author ="S. {YU.} Favorov",
    title = "LAGRANGE’S MEAN MOTION PROBLEM",
    journal = "St. Petersburg Math. J.",
    year = "2008", volume="20",number="2", pages="319-324",
}

@book{bonnard_chyba,
    author = "Bernard Bonnard 
and Monique Chyba",
    title ="Singular Trajectories and their Role in Control Theory" ,
    publisher = "Springer",
    year = "2003",
}

@book{sontag_book_control_theory,
    author ="E. D. Sontag" ,
    title = "Mathematical Control Theory: 
Deterministic Finite Dimensional Systems",
    publisher ="Springer" ,
    year = "1998", address="New York, NY"
}

@article{Wintner-1933,
    author = "A. Wintner",
    title = "Upon a Statistical Method in the Theory of Diophantine Approximations",
    journal = "Amer. J. Math.",
    year = "1933", pages="309-331", volume="55", number="1", 
}

@article{margaliot_liberzon_lie_algebraic,
title = {Lie-algebraic stability conditions for nonlinear switched systems and differential inclusions},
journal =SCL,
volume = {55},
number = {1},
pages = {8-16},
year = {2006},
 author = {Michael Margaliot and Daniel Liberzon},
 }

@incollection{levinson_select,
  author      ="J. Nohel and H. J. Sussmann",
  title       = "{Commentary on Minimax, Liapunov and Bang-Bang}",
  editor      = "J A. Nohel and D. H. Sattinger",
  booktitle   = "Selected Papers of Norman Levinson",
  publisher   = "Birkhauser",
  address     = "Boston",
  year        = "1998",
  pages       = "463-475",
   volume="2", 
}

@book{real_analysis_stromberg,
    author = "Karl R. Stromberg",
    title = "Introduction to Classical Real Analysis",
    publisher = "Wadsworth Inc.",
    year = "1981", 
    address="Belmont, CA",
}

@article{SHARON_3_nilp,
title = {Third-order nilpotency, nice reachability and asymptotic stability},
journal = JDE,
volume = {233},
number = {1},
pages = {136-150},
year = {2007},
author = {Yoav Sharon and Michael Margaliot},
}

\end{document}